\UseRawInputEncoding 
\documentclass[11pt]{article}
\usepackage{latexsym}
\usepackage{mathrsfs}
\usepackage{amsmath,amsthm}
\usepackage{graphicx}
\usepackage{epsfig}
\usepackage{epstopdf}
\usepackage{amssymb}
\usepackage{CJK}
\usepackage{color}
\usepackage{subcaption}
\usepackage{mathabx}
\usepackage{indentfirst}

\newtheorem{thm}{Theorem}[section]

\newtheorem{lem}[thm]{Lemma}

\newtheorem{alg}[thm]{Algorithm}
\DeclareGraphicsExtensions{.eps,.eps.gz}
\normalsize \rm
\begin{document}
\title{The number of maximal matchings in polygon rings}
\author{Chengqi Li$^{a}$, Jinhui Yin$^{a}$, Lingjuan Shi$^a$\footnote{Corresponding author.}
\date{\small $^{a}$ School of Software, Northwestern Polytechnical University, Xi'an, Shaanxi 710072, P. R. China\\
E-mails: lichengqi@mail.nwpu.edu.cn, jinhuiyin@mail.nwpu.edu.cn, shilj18@nwpu.edu.cn}
}

\maketitle
\begin{abstract}
	A matching of graph $G$ is maximal if it cannot be expanded by adding any edge to create a larger matching.
In this paper, for a hexagonal ring $H$ with $n$ hexagons, we show that the number of maximal matchings of $H$
equals to the trace of the product of $n$ matrices, each of which is $S$, $L$, or $R$ according to the type of the connection mode of $H$.
Finally, we extend this conclusion to arbitrary polygon rings and provide an algorithm to determine the transition matrices of polygon chains (rings).
	
\end{abstract}
\baselineskip=0.3in

\section{Introduction and Preliminaries}

A \emph{polygon chain} $Z$ with $n$ faces is a $2$-connected planar graph composed of $n$ cycles
each of which has length at least $4$, and all of these cycles are attached together in a linear order and each vertex of $Z$ is shared by at most two such cycles. These $n$ cycles are called \emph{faces} of $G$, and two faces are \emph{adjacent} if they share exactly one edge.
Here we note that any two faces in $Z$ share at most one edge, and any edge in shared by at most two faces. An example is depicted in Fig. \ref{fig:polygon-ring}(a).
\begin{figure}[h]
\centering
\includegraphics[width=0.9\linewidth]{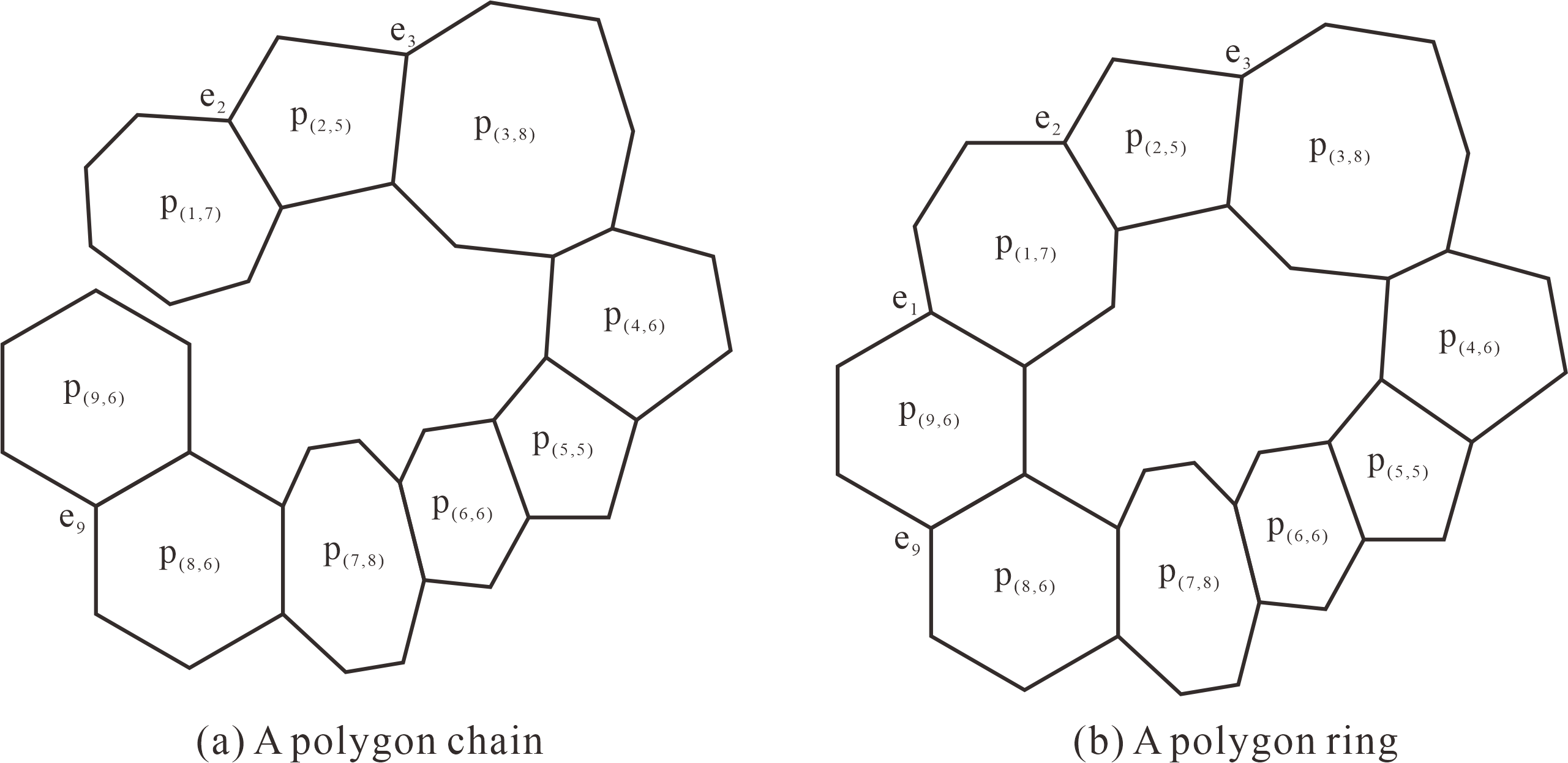}
\caption{Examples for polygon chain and polygon ring.}
\label{fig:polygon-ring}
\end{figure}

Let $Z$ be a polygon chain with $n$ faces, $n\geq2$ . We label the $n$ faces of $Z$ successively as $p_{(1,s_1)}, p_{(2,s_2)},\ldots,p_{(n,s_n)}$, where $p_{(1,s_1)}$ and $p_{(n,s_n)}$ are the two terminal faces, and the $i$th face $p_{(i,s_i)}$ is an $s_i$-length cycle.
For $i\in [2,n]$, $p_{(i-1,s_{i-1})}$ and $p_{(i,s_i)}$ are adjacent, and we denote the edge shared by them as $e_i$, $i\in [2,n]$.
For any $j\in [2,n-1]$, $p_{(j,s_j)}$ shares $e_j$ with $p_{(j-1,s_{j-1})}$ and shares $e_{j+1}$ with $p_{(j+1,s_{j+1})}$.
Along the clockwise direction of face $p_{(j,s_j)}$, if there are exactly $k_j$ edges between $e_j$ and $e_{j+1}$, then we denote the connection type of $p_{(j+1,s_{j+1})}$ to $p_{(j,s_j)}$ by $t(s_{j},k_j)$, where $k_j\in[1,s_{j}-3]$.
Now, we can more precisely denote a polygon ring with $n$ faces by $t(s_{1},*)t(s_{2},k_2)\cdots t(s_{n-1},k_{n-1})t(s_{n},*)$.
See Fig. \ref{fig:polygon-ring}(a) for a polygon chain $t(7,*)t(5,2)t(8,4)t(6,3)t(5,2)t(6,2)t(8,3)t(6,3)t(6,*)$ with $9$ faces. Clearly, each polygon chain has a unique such representation if we start to read from one fix terminal hexagon.

By attaching the two faces $p_{(1,s_1)}$ and $p_{(n,s_n)}$ together, we obtain a graph, which is called a polygon
ring if it can be embedded on the plane.
More specifically,
if we glue a $2$-edge (the two incident vertices both have degree $2$ in the original polygon chain) of face $p_{(1,s_1)}$ with a $2$-edge of face $p_{(n,s_n)}$ such that the obtained new graph is planar, then the new graph is called a \emph{polygon ring}, and
the adhesive edge is denoted by $e_1$.
As the above description of polygon chain, a polygon ring with $n$ faces can be denoted by $t(s_{1},k_1)t(s_{2},k_2)\cdots t(s_{n-1},k_{n-1})t(s_{n},k_n)$, where $k_1$ ($k_n$ resp.) is the number of edges from $e_1$ to $e_{2}$  (from $e_n$ to $e_1$) along the clockwise direction of face $p_{(1,s_1)}$ ($p_{(n,s_n)}$ resp.).
Fig. \ref{fig:polygon-ring}(b) is a polygon ring $t(7,3)t(5,2)t(8,4)t(6,3)t(5,2)t(6,2)t(8,3)t(6,3)t(6,3)$ corresponding to the polygon chain in Fig. \ref{fig:polygon-ring}(a).

If all the faces of a polygon chain are hexagons, then it is a \emph{hexagonal chain}. Accordingly, we can get a \emph{hexagonal ring} (see Fig. \ref{fig:hexagon-chain-ring} for an example).
Here, we need to note that in corresponding to reference \cite{shi2023counting},
the type of the $i$-th face in a hexagonal chain is $t(6,2)$ if the next hexagon $h_{i+1}$ is linearly adhesive, is $t(6,1)$ if the next hexagon $h_{i+1}$ is left angularly adhesive, is $t(6,3)$ if the next hexagon $h_{i+1}$ is right angularly adhesive.
Since all faces of hexagonal chains and hexagonal rings are $6$-length cycles, the symbol $s_j$ in the labels $p(j,s_j)$ and $t(s_j,k_j)$ as depicted above can be omitted. And for a hexagonal chain with $n$ hexagons, we can delete its first and last signs from the above definition of type, and its type can be denoted by $M_2M_3\cdots M_{n-1}$ where $M_i=t(1), t(2)$ or $t(3)$.
For example, the type of the hexagon chain in Fig. \ref{fig:hexagon-chain-ring}(a) is $t(3)t(3)t(1)t(3)t(3)t(3)t(2)t(2)t(3)$,
and correspondingly, the type of the hexagon ring in Fig. \ref{fig:hexagon-chain-ring}(b) is $t(2)t(3)t(3)t(1)t(3)t(3)t(3)t(2)t(2)t(3)t(3)$.\\
\begin{figure}[h]
\centering
\includegraphics[width=0.9\linewidth]{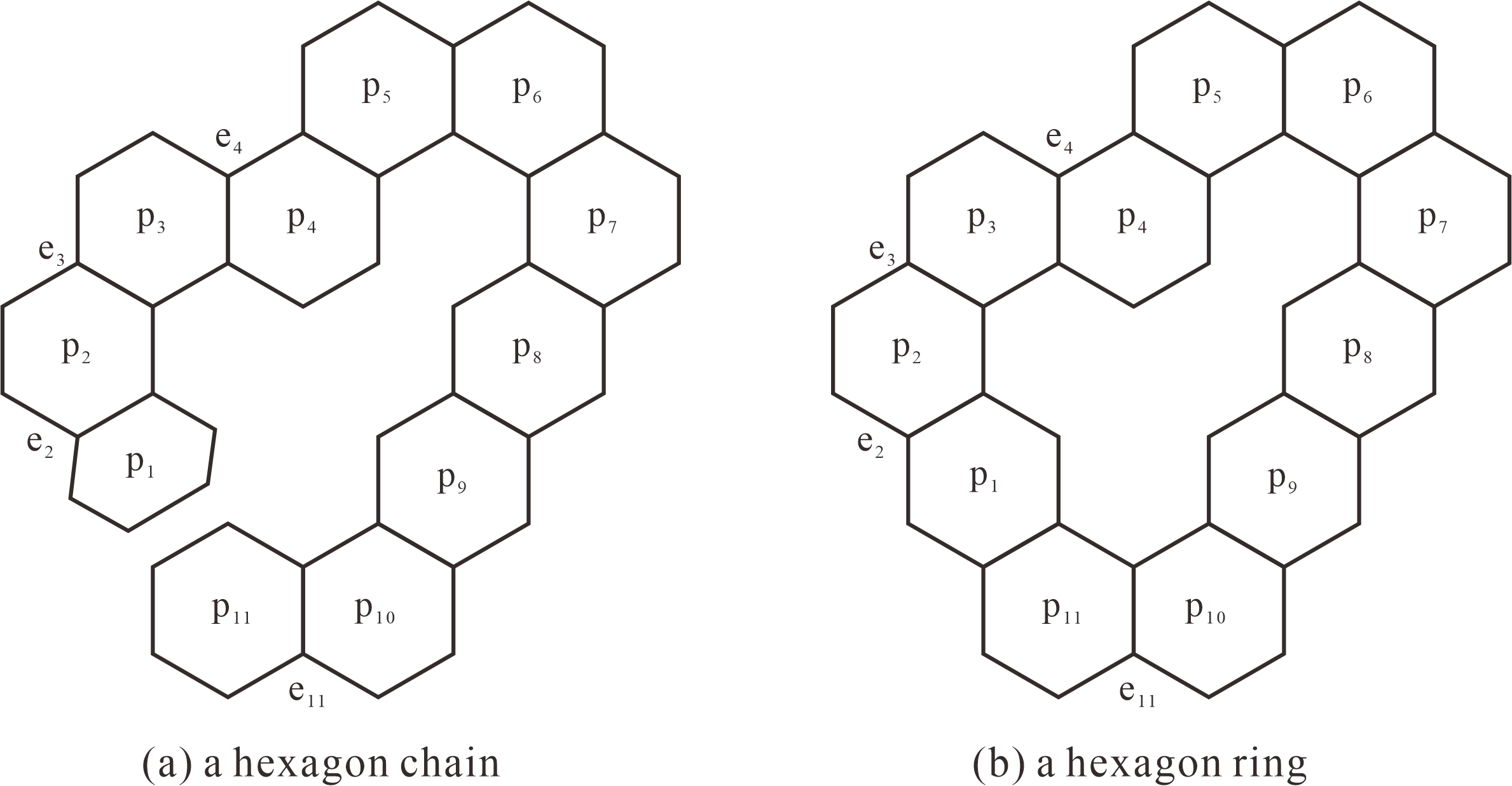}
\caption{Examples for hexagon chain and hexagon ring.}
\label{fig:hexagon-chain-ring}
\end{figure}

Hexagonal rings which is composed of several regular hexagons are molecular graphs of cyclic aromatic hydrocarbons called Cycloarenes \cite{staab1983cycloarenes}.
Cyclic aromatic hydrocarbons, such as annulenes, polycyclic aromatic hydrocarbons, and carbon nanobelts (CNBs), have attracted extensive attention due to their unique structural and electronic features.
Studying the topological structure of cyclic aromatic hydrocarbons is crucial for understanding the properties of the resulting materials.

A \emph{matching} $M$ of graph $G$ is a subset of edges, and any two edges in $M$ do not share a common vertex.
The number of edges in $M$ is referred to as its \emph{size}.
If a vertex of $G$ is incident with an edge in $M$, then we say that
the vertex is \emph{covered} by $M$, otherwise, \emph{uncovered} by $M$.
The \emph{maximum matching} of $G$ is a matching having the maximum number of edges.
A matching $M$ is considered \emph{maximal} if it does not contained in any larger matching of $G$.
We use the notation $\Psi(G)$ to denote the number of all the maximal matchings in graph $G$.
A matching covering all vertices of graph $G$ is called \emph{perfect}.
It can be observed that both maximum and perfect matchings are always maximal, but the reverse is not necessarily true.
So far, there have been many research achievements regarding maximum matching and perfect matching (see \cite{cyvin1988kekule}, \cite{galil1986efficient}, \cite{lovasz1986matching} etc.), but there are few results concerning maximal matching.

Maximal matching serves as a model for many physical and technical problems, such as block allocation of sequential resources or adsorption of dimers on structured substrates or molecules.
T.~Do{\v{s}}li{\'c} et al. (see \cite{dovslic2021maximal},\cite{dovslic2015saturation},\cite{dovslic2016counting}) obtained some results on benzene-type graphs and linear polymers.
J.~G{\'o}rska and Z.~Skupie{\'n} \cite{gorska2007trees}, M.~Klazar \cite{klazar1997twelve} have counted the number of maximal matchings in trees.
Inspired by the method of the transition matrix and Hosoya vector of M. Oz et al. \cite{oz2022computing} and R. Cruz et al. \cite{cruz2017computing}, the corresponding author of this article and Deng \cite{shi2023counting} using maximal matching vectors obtained the maximal matching numbers of hexagon chains.

In this paper, we discuss the counting problems of maximal matchings in any polygon rings.
And based on the maximal matching vector and transfer matrix technique, we obtain a specific formula for calculating the number of maximal matchings in polygon rings, especially in the hexagonal rings.
The following three matrices have important applications in the future and is depicted in reference \cite{shi2023counting}.

\setlength{\arraycolsep}{1.2pt}
\[
S=\begin{pmatrix}
        1&1&1&1&1&1&1&0&0\\
		0&1&0&1&0&0&1&1&0\\
		0&0&1&1&0&1&0&0&1\\
		0&0&0&1&1&1&1&0&0\\
		1&0&0&1&1&1&1&0&0\\
		0&1&0&0&0&0&0&1&0\\
		0&0&1&0&0&0&0&0&1\\
		1&0&1&1&1&1&1&0&0\\
		1&1&0&1&1&1&1&0&0\\
\end{pmatrix},\qquad
L=\begin{pmatrix}
        1&1&1&1&1&1&1&0&0\\
		1&0&1&0&0&0&0&0&1\\
		0&1&0&1&0&0&1&1&0\\
		1&0&1&0&0&0&0&0&0\\
		1&0&1&1&0&1&0&0&0\\
		0&0&1&0&0&0&0&0&1\\
		0&1&0&1&0&0&0&0&0\\
		1&1&1&1&0&1&0&0&0\\
		1&0&1&1&1&1&1&0&0\\
\end{pmatrix},\qquad
R=\begin{pmatrix}
        1&1&1&1&1&1&1&0&0\\
		0&0&1&1&0&1&0&0&1\\
		1&1&0&0&0&0&0&1&0\\
		1&1&0&0&0&0&0&0&0\\
		1&1&0&1&0&0&1&0&0\\
		0&0&1&1&0&0&0&0&0\\
		0&1&0&0&0&0&0&1&0\\
		1&1&0&1&1&1&1&0&0\\
		1&1&1&1&0&0&1&0&0\\
\end{pmatrix}.
\]

The main conclusions of this paper are listed below.
\begin{thm}\label{Hexagond}
Let $H$ be a hexagonal ring with type $M_1M_2\cdots M_n$, where $M_i=t(1), t(2)$ or $t(3)$. Then the number of maximal matchings of $H$ can be computed by the following formula:
$$\Psi(H)=tr(\prod\limits_{i=1}^{n}f(M_i)),$$
	where $f(M_i)=L$ if $M_i=t(1)$, $f(M_i)=S$ if $M_i=t(2)$, $f(M_i)=R$ if $M_i=t(3)$, the matrix $L, S, R$ are given as above and $tr(A)$ is the trace of a matrix $A$.
\end{thm}
Applying Theorem \ref{Hexagond}, for the hexagon ring $G$ as depicted in Fig. \ref{fig:hexagon-chain-ring} (b), we have
$\Psi(G)=tr(S \times R \times R \times L \times R \times R \times R \times S \times S \times R \times R )= 2804280 $.

In the following, for the sake of simplicity,
we denote the product of these $n$ matrices in Theorem \ref{Hexagond} by $M:=\prod\limits_{i=1}^{n}(f(M_i))$.
Furthermore, the Theorem \ref{Hexagond} can be extended to any polygon ring by the same method.
\begin{thm}\label{Ploygon}
	For any polygon ring $Z$ composed of $n$ polygons with type $M_1M_2\cdots M_n$, where $M_i\in\{t(s_i,1), t(s_i,2), \ldots t(s_i,s_i-4), t(s_i, s_i-3)$, the number of maximal matchings of $Z$ is
	$$\Psi(Z)=tr(\prod\limits_{i=1}^{n}f(M_i)),$$
	where $f(M_{i})=T_{s_i,1}$ if $M_i=t(s_i,1)$, $f(M_{i})=T_{s_i,2}$ if $M_i=t(s_i,2)$, $\ldots$, $f(M_{i})=T_{s_i,s_i-3}$ if $M_i=t(s_i,s_i-3)$, and the matrix $T_{s_i,j}$ can be easily computed by the Algorithm \ref{algorithm}, $j\in\{1, 2\cdots, s_i-3\}$.
\end{thm}

\section{Hexagonal rings}

\indent
Now, let's start to introduce some symbols that will be used later.
For a vertex $v$ in graph $G$, the \emph{neighborhood} of $v$ in $G$ is defined as $N_{G}(v)=\{u\in G| uv\in E(G)\}$.
$\Psi(G|a)$ denotes the number of maximal matchings in graph $G$ with $a$ must being covered.
$\Psi(G|bc)$ denotes the numbers of maximal matchings in graph $G$ with edge $bc$ must being contained.
Obviously, we have
$$\Psi(G|a)=\sum\limits_{y\in N_G(a)}\Psi(G-a-y),$$
$$\Psi(G|bc) = \Psi(G-b-c).$$
$\Psi^{-a}(G)$ denotes the number of maximal matchings in graph $G$ with vertex $a$ not covered.
$\Psi^{-bc}(G)$ denotes the number of maximal matchings in graph $G$ with the edge $bc$ avoided.
It should be noted that $\Psi^{-a}(G)$ is not equivalent to $\Psi(G-a)$, and $\Psi^{-bc}(G)$ is not equivalent to $\Psi(G-bc)$.
The two formers are the numbers of maximal matchings in graph $G$ with the asked property, while the two latters are the numbers of maximal matchings in the subgraph obtained by removing point $a$ or edge $bc$ from graph $G$.
$Mat[i,j]$ represents the value of the element in the $i$-th row and $j$-th column of matrix $Mat$. $Vec_{[i]}$ represents the $i$-th component of vector $Vec$.
The \emph{maximal matching vector} $\Psi_{xy}(G)$ for the edge $xy$ in graph $G$ is defined as follows:
\begin{equation}\nonumber
	\Psi_{xy}(G)=
	\left(
	\begin{array}{c}
		\Psi(G)\\
		\Psi(G-x)\\
		\Psi(G-y)\\
		\Psi(G-x-y)\\
		\Psi(G|x,y)\\
		\Psi(G-x|y)\\
		\Psi(G-y|x)\\
		\Psi(G|x)\\
		\Psi(G|y)\\
	\end{array}
	\right).
\end{equation}
We note that the edge $bc$ and vertex $a$ in the above notations $\Psi(G|a), \Psi(G|bc), \Psi^{-a}(G), \Psi^{-bc}(G)$ can be any set which is the union of several vertices and several edges.
For example, $\Psi^{-\{a,ed\}}(G-b|f,xy)$ is the number of maximal matchings in graph $G-b$, each of which covers vertex $f$, contains edge $xy$, and not cover vertex $a$, not contain edge $ed$.
Moreover, the $9$ components of the vector $\Psi_{uv}^{-\{a,ed\}}(G-b|f,xy)$ are $\Psi^{-\{a,ed\}}(G-b|f,xy)$, $\Psi^{-\{a,ed\}}(G-b-u|f,xy)$, $\Psi^{-\{a,ed\}}(G-b-v|f,xy)$, $\Psi^{-\{a,ed\}}(G-b-u-v|f,xy)$, $\Psi^{-\{a,ed\}}(G-b|u,v,f,xy)$, $\Psi^{-\{a,ed\}}(G-b-u|v,f,xy)$, $\Psi^{-\{a,ed\}}(G-b-v|u,f,xy)$, $\Psi^{-\{a,ed\}}(G-b|u,f,xy)$, $\Psi^{-\{a,ed\}}(G-b|v,f,xy)$ respectively.
It is not hard to check that $\Psi^{-\{bc,b\}}(G)=\Psi^{-b}(G)$, $\Psi^{-bc}(G|b)=\Psi(G-bc|b)$, which will be applied repeatedly without explanation in the following.
\begin{figure}[h]
\centering
\includegraphics[width=0.2\linewidth]{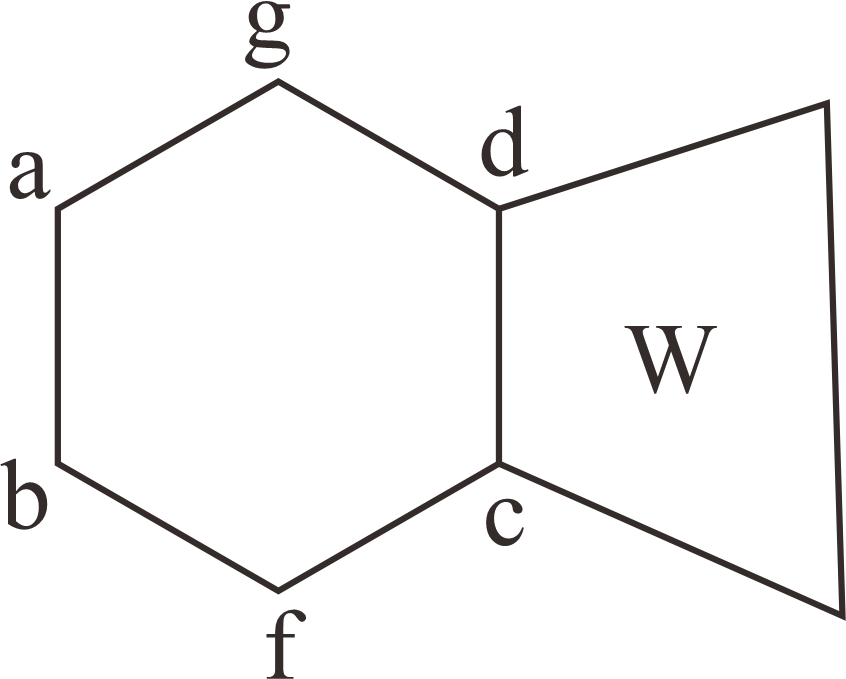}
\caption{The graph in Theorem \ref{shi-deng-process}.}\label{fig:hexagon-process}
\end{figure}

The following two Theorems computed the cases for hexagonal chains.
\begin{thm}[\cite{shi2023counting}]\label{shi-deng-process}
For the graph $G$ as depicted in Fig. \ref{fig:hexagon-process}, we have
$$\Psi_{ab}(G)=S\times\Psi_{dc}(W),~~\Psi_{ga}(G)=L\times\Psi_{dc}(W),~~\Psi_{bf}(G)=R\times\Psi_{dc}(W).$$
\end{thm}

\begin{thm}[\cite{shi2023counting}]\label{shi-deng}
	For any benzenoid chain $H$ with $n$ hexagons and having type $M_2\cdots M_{n-1}$, $M_i\in\{t(1), t(2), t(3))\}$, we have
$$\Psi(H)=X\times S\times\prod\limits_{i=2}^{n-1}(f(M_i))\times S\times Y,$$
 where $f(M_i)=L$ if $M_i=t(1)$, $f(M_i)=S$ if $M_i=t(2)$, $f(M_i)=R$ if $M_i=t(3)$ and $X=(1,0,0,0,0,0,0,0,0)$, $Y=(1,1,1,1,1,0,0,1,1)^T$.
\end{thm}
See the Theorem $2.6$ in reference \cite{shi2023counting}, clearly, the vectors $X\times S$ and $S\times Y$ in the above Theorem \ref{shi-deng} have $X\times S=(1,1,1,1,1,1,1,0,0)$, $S\times Y=(5,3,3,2,3,2,2,4,4)$.

For graph $G$ which is the union of two disjoint graphs $G_1$ and $G_2$, obviously,
any maximal matching of $G$ is consisted of a maximal matching of $G_1$ and a maximal matching
of $G_2$. So we have $\Psi(G)=\Psi(G_1)\times\Psi(G_2)$.
This property will be applied repeatedly.

Now, we are ready to study the hexagonal ring.
As illustrated in Fig. \ref{Auxiliary graph}, graph $H$ is a hexagonal ring formed by $n$ hexagons $H_1, H_2, ..., H_n$ in a sequence, with $H_1$ and $H_n$ sharing edge $dc$.
As depicted in Fig. \ref{Auxiliary graph}, graph $G$ is obtained from $H$ by tearing $H$ along the edge $dc$, the vertex corresponding to $d$ in hexagon $H_1$ is named as $a$, and the vertex corresponding to $c$ is named as $b$.
Graph $F$ serves as an auxiliary graph obtained from $G$ by glue an arbitrary subgraph $K$ to the edge $dc$, see Fig. \ref{Auxiliary graph}.
\begin{figure}[h]
\centering
\includegraphics[width=1.0\linewidth]{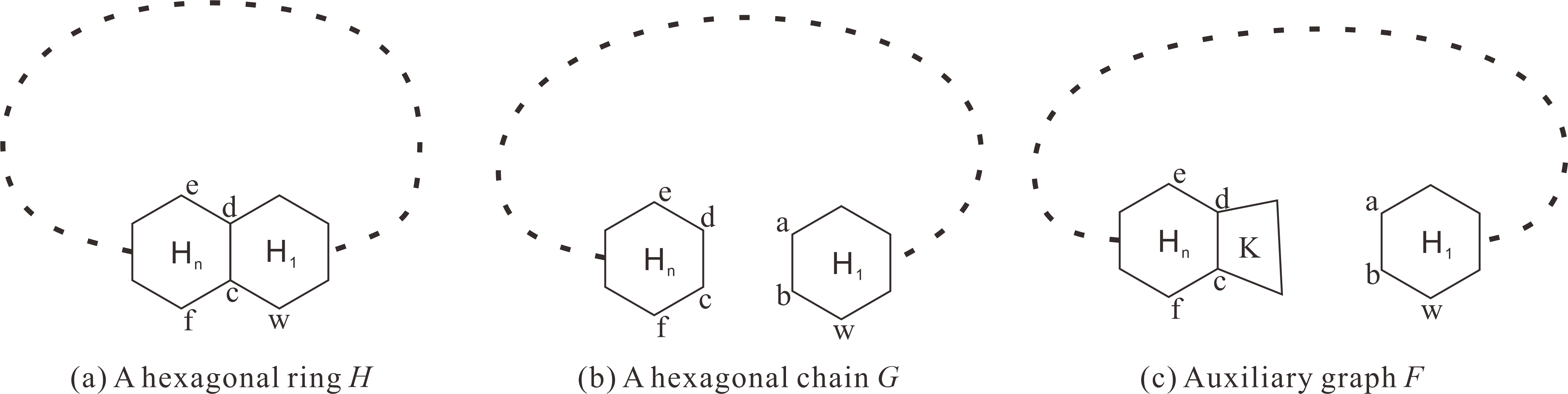}
\caption{Graphs $H$, $G$ and auxiliary graph $F$}\label{Auxiliary graph}	
\end{figure}

In the following, the graphs $H$, $F$ and $G$ are as depicted in Fig. \ref{Auxiliary graph}, and we suppose that the type of $H$ is $M_1M_2\cdots M_n$, where $M_i=t(1), t(2)$ or $t(3)$.
Clearly, $G$ is a hexagonal chain having type $M_2M_3\cdots M_{n-1}$. If we set edge $ab$ as the first edge in $H_1$, and $dc$ as the last edge in $H_n$, then the type of $G$ can also be denoted by $M_1M_2\cdots M_n$.

Applying Theorem \ref{shi-deng-process} repeatedly, we can easily know that $\Psi_{ab}(F)=M\times \Psi_{dc}(K)$ (recall that $M:=\prod\limits_{i=1}^{n}(f(M_i))$),
and moreover the following equation holds for any $i\in\{1,2,\ldots,9\}$.
\begin{multline}\label{1.0}
	[\Psi_{ab}(F)]_{[i]} = M[i,1]\cdot\Psi(K)+M[i,2]\cdot\Psi(K-d)+M[i,3]\cdot\Psi(K-c)+M[i,4]\cdot\Psi(K-d-c)\\
    +M[i,5]\cdot\Psi(K|d,c)+M[i,6]\cdot\Psi(K-d|c)+M[i,7]\cdot\Psi(K-c|d)+M[i,8]\cdot\Psi(K|d)+M[i,9]\cdot\Psi(K|c).
\end{multline}

\subsection{\textbf{Nine Key Lemmas}}
Depending on whether edges $ed$ and $fc$ are included in the maximal matching of $F$, all the maximal matchings of $F$ can be decomposed into nine cases, corresponding to the following nine Lemmas \ref{1.1}-\ref{1.9}.
\begin{lem}\label{1.1}
	If both $ed$ and $fc$ belong to the maximal matching of $F$, then for $i=1,2,\ldots, 9$ we have
	$$[\Psi_{ab}(F|ed,fc)]_{[i]}=M[i,4]\cdot\Psi(K-c-d).$$
\end{lem}
\begin{proof}
By the definition of maximal matching vector, we have $\Psi_{ab}(F|ed,fc)=(\Psi(F|ed,fc),$
$\Psi(F-a|ed,fc),\Psi(F-b|ed,fc),\Psi(F-a-b|ed,fc),\Psi(F|ed,fc,a,b),\Psi(F-a|ed,fc,b),\Psi(F-b|ed,fc,a),$
$\Psi(F|ed,fc,a),\Psi(F|ed,fc,b))^T=(\Psi(F-e-d-f-c),\Psi(F-e-d-f-c-a),\Psi(F-e-d-f-c-b),\Psi(F-e-d-f-c-a-b),\Psi(F-e-d-f-c|a,b), \Psi(F-e-d-f-c-a|b),\Psi(F-e-d-f-c-b|a),\Psi(F-e-d-f-c|a),\Psi(F-e-d-f-c|b))^T$.
Now, we compute each component of the vector $\Psi_{ab}(F|ed,fc)$ as follows.

Since $F-e-d-f-c$ consists of two components $G-e-d-f-c$ and $K-d-c$ (see Fig. \ref{Auxiliary graph}), $[\Psi_{ab}(F|ed,fc)]_{[1]}=\Psi(F-e-d-f-c)=\Psi(G-e-d-f-c)\times\Psi(K-d-c)$.
We recall that the faces of $G$ are denoted by $H_1, H_2, \ldots, H_{n}$ successively.
Since $\Psi(G-e-d-f-c)=(1,0,0,0,0,0,0,0,0)\times\Psi_{ab}(G-e-d-f-c)=(1,0,0,0,0,0,0,0,0)\times f(M_1)\times f(M_2)\times\cdots\times f(M_{n-2})\times\Psi_{xy}(D)$, where $xy$ is the common edge of faces $H_{n-2}$ and $H_{n-1}$ and it is clockwise from $x$ to $y$ on $H_{n-2}$.
We note that $D$ is a subgraph of $G$, which is the union of $H_{n-1}$ and $H_{n}$ and delete vertices $e,d,f,c$.
For the sake of arguments, we denote the common edge sharing by faces $H_{n-1}$ and $H_{n}$ by $uv$, and let $uv$ be clockwise from $u$ to $v$ on $H_{n-1}$.

If $u=e$, then it is easy to check that $\Psi_{xy}(D)=f(M_{n-1})\times L\times(0,0,0,1,0,0,0,0,0)^T=f(M_{n-1})\times f(M_n)\times(0,0,0,1,0,0,0,0,0)^T$.

If $\{u,v\}\cap\{e,f\}=\emptyset$, then $\Psi_{xy}(D)=f(M_{n-1})\times S\times(0,0,0,1,0,0,0,0,0)^T=f(M_{n-1})\times f(M_n)\times(0,0,0,1,0,0,0,0,0)^T$.

If $v=f$, then $\Psi_{xy}(D)=f(M_{n-1})\times R\times(0,0,0,1,0,0,0,0,0)^T=f(M_{n-1})\times f(M_n)\times(0,0,0,1,0,0,0,0,0)^T$.

So $\Psi_{xy}(D)=f(M_{n-1})\times f(M_{n})\times(0,0,0,1,0,0,0,0,0)^T$.
To sum up, $[\Psi_{ab}(F|ed,fc)]_{[1]}=(1,0,0,0,0,0,0,0,0)\times f(M_1)\times f(M_2)\times\cdots\times f(M_{n-2})\times f(M_{n-1})\times f(M_n)\times(0,0,0,1,0,0,0,0,0)^T\times\Psi(K-d-c)=M[1,4]\cdot\Psi(K-c-d)$.

By the same analysis, if we change $(1,0,0,0,0,0,0,0,0)$ to $(0,1,0,0,0,0,0,0,0)$ in the above process, then we can show that $[\Psi_{ab}(F|ed,fc)]_{[2]}=\Psi(F-e-d-f-c-a)=M[2,4]\cdot\Psi(K-c-d)$.
For $i=3,\ldots, 9$, we only need to change $(1,0,0,0,0,0,0,0,0)$ to the corresponding vectors in the above process. Then we are done.
\end{proof}

\begin{lem}\label{1.2}
	If $ed$ is in the maxiaml matching and $fc$ is not in, but $f$ is covered, then for $i=1,2,\ldots, 9$ we have
	$$[\Psi_{ab}^{-fc}(F|ed,f)]_{[i]}=M[i,2]\cdot\Psi(K-d).$$
\end{lem}
\begin{proof}
As the proof of Lemma \ref{1.1}, we have $\Psi_{ab}^{-fc}(F|ed,f)=\Psi_{ab}(F-e-d-fc|f)$.
We compute each component of the vector $[\Psi_{ab}^{-fc}(F|ed,f)]$ as follows.

Since $F-e-d-fc$ consists of two components $G-e-d-c$ and $K-d$,
$[\Psi_{ab}^{-fc}(F|ed,f)]_{[1]}=\Psi(F-e-d-fc|f)=\Psi(G-e-d-c|f)\times\Psi(K-d)$.
Since $\Psi(G-e-d-c|f)=(1,0,0,0,0,0,0,0,0)\times\Psi_{ab}(G-e-d-c|f)=(1,0,0,0,0,0,0,0,0)\times f(M_1)\times f(M_2)\times\cdots\times f(M_{n-2})\times\Psi_{xy}(D|f)$, where $xy$ as described in Lemma \ref{1.1}.
We note that $D$ is a subgraph of $G$, which is the union of $H_{n-1}$ and $H_{n}$ and delete vertices $e,d,c$.
The common edge shared by faces $H_{n-1}$ and $H_{n}$ is denoted by $uv$ (as described in the proof of Lemma \ref{1.1}).

If $u=e$, then it is easy to check that $\Psi_{xy}(D|f)=f(M_{n-1})\times L\times(0,1,0,0,0,0,0,0,0)^T=f(M_{n-1})\times f(M_n)\times(0,1,0,0,0,0,0,0,0)^T$.

If $\{u,v\}\cap\{e,f\}=\emptyset$, then $\Psi_{xy}(D)=f(M_{n-1})\times S\times(0,1,0,0,0,0,0,0,0)^T=f(M_{n-1})\times f(M_n)\times(0,1,0,0,0,0,0,0,0)^T$.

If $v=f$, then $\Psi_{xy}(D)=f(M_{n-1})\times R\times(0,1,0,0,0,0,0,0,0)^T=f(M_{n-1})\times f(M_n)\times(0,1,0,0,0,0,0,0,0)^T$.

So $[\Psi_{ab}^{-fc}(F|ed,f)]_{[1]}=M[1,2]\cdot\Psi(K-d)$.
By the same analysis,
for $i=2,\ldots, 9$, we only need to change $(1,0,0,0,0,0,0,0,0)$ to the corresponding vectors in the above process. Then we are done.
\end{proof}

\begin{lem}\label{1.3}
	If $ed$ is in the maximal matching, $fc$ is not in, and $f$ is not covered, for $i=1,2,\ldots, 9$ we have
	$$[\Psi_{ab}^{-\{f,fc\}}(F|ed)]_{[i]}=M[i,6]\cdot\Psi(K-d|c).$$
\end{lem}
\begin{proof}
As the proof of Lemma \ref{1.1}, we have $\Psi_{ab}^{-\{f,fc\}}(F|ed)=\Psi_{ab}^{-f}(F|ed)=\Psi_{ab}^{-f}(F-e-d)$, each component of which is computed as follows.

$[\Psi_{ab}^{-\{f,fc\}}(F|ed)]_{[1]}=\Psi^{-f}(F-e-d)=\Psi^{-f}(G-e-d-c)\times\Psi(K-d|c)$.
Since $\Psi^{-f}(G-e-d-c)=(1,0,0,0,0,0,0,0,0)\times\Psi_{ab}^{-f}(G-e-d-c)=(1,0,0,0,0,0,0,0,0)\times f(M_1)\times f(M_2)\times\cdots\times f(M_{n-2})\times\Psi_{xy}^{-f}(D)$, where $xy$ is the same as described in the proof of Lemma \ref{1.1}.
We note that $D$ is a subgraph of $G$, which is the union of $H_{n-1}$ and $H_{n}$ and delete vertices $e,d,c$.

For the three cases of the sharing edge between $H_{n-1}$ and $H_{n}$, we can check that
$\Psi_{xy}^{-f}(D)=f(M_{n-1})\times f(M_{n})\times(0,0,0,0,0,1,0,0,0)^T$.
So $[\Psi_{ab}^{-\{f,fc\}}(F|ed)]_{[1]}=M[1,6]\cdot\Psi(K-d|c)$.
By the same analysis,
for $i=2,\ldots, 9$, we only need to change $(1,0,0,0,0,0,0,0,0)$ to the corresponding vectors in the above process. Then we are done.
\end{proof}

\begin{lem}\label{1.4}
	If $fc$ is in the maximal matching, $ed$ is not in, and $e$ is covered, then
	$$[\Psi_{ab}^{-ed}(F|e,fc)]_{[i]}=M[i,3]\cdot\Psi(K-c).$$
\end{lem}
\begin{proof}
As the proof of Lemma \ref{1.1}, we have $\Psi_{ab}^{-ed}(F|e,fc)=\Psi_{ab}(F-f-c-ed|e)$, each component of which is computed as follows.

$[\Psi_{ab}^{-ed}(F|e,fc)]_{[1]}=\Psi(F-f-c-ed|e)=\Psi(G-f-c-d|e)\times\Psi(K-c)$.
Since $\Psi(G-f-c-d|e)=(1,0,0,0,0,0,0,0,0)\times\Psi_{ab}(G-f-c-d|e)=(1,0,0,0,0,0,0,0,0)\times f(M_1)\times f(M_2)\times\cdots\times f(M_{n-2})\times\Psi_{xy}(D|e)$, where $xy$ is the same as described in the proof of Lemma \ref{1.1}.
We note that $D$ is a subgraph of $G$, which is the union of $H_{n-1}$ and $H_{n}$ and delete vertices $f,c,d$.

For the three cases of the sharing edge between $H_{n-1}$ and $H_{n}$, we can check that
$\Psi_{xy}(D|e)=f(M_{n-1})\times f(M_{n})\times(0,0,1,0,0,0,0,0,0)^T$.
So $[\Psi_{ab}^{-ed}(F|e,fc)]_{[1]}=M[1,3]\cdot\Psi(K-c)$.
By the same analysis,
for $i=2,\ldots, 9$, we only need to change $(1,0,0,0,0,0,0,0,0)$ to the corresponding vectors in the above process. Then we are done.
\end{proof}

\begin{lem}\label{1.5}
	If $fc$ is in the maximal matching, $ed$ is not in, and $e$ is not covered, then
	$$[\Psi_{ab}^{-\{e,ed\}}(F|cf)]_{[i]}=M[i,7]\cdot\Psi(K-c|d).$$
\end{lem}
\begin{proof}
As the proof of Lemma \ref{1.1}, we have $\Psi_{ab}^{-\{e,ed\}}(F|cf)=\Psi_{ab}^{-e}(F|cf)=\Psi_{ab}^{-e}(F-f-c)$, each component of which can be computed as follows.

$\Psi_{ab}^{-e}(F-f-c)_{[1]}=\Psi^{-e}(F-f-c)=\Psi^{-e}(G-f-c-d)\times\Psi(K-c|d)$.
Since $\Psi^{-e}(G-f-c-d)=(1,0,0,0,0,0,0,0,0)\times\Psi_{ab}^{-e}(G-f-c-d)=(1,0,0,0,0,0,0,0,0)\times f(M_1)\times f(M_2)\times\cdots\times f(M_{n-2})\times\Psi_{xy}^{-e}(D)$, where $xy$ is the same as described in the proof of Lemma \ref{1.1}.
We note that $D$ is a subgraph of $G$, which is the union of $H_{n-1}$ and $H_{n}$ and delete vertices $f,c,d$.

For the three cases of the sharing edge between $H_{n-1}$ and $H_{n}$, we can check that
$\Psi_{xy}^{-e}(D)=f(M_{n-1})\times f(M_{n})\times(0,0,0,0,0,0,1,0,0)^T$.
So $[\Psi_{ab}^{-e}(F|cf)]_{[i]}=M[i,7]\cdot\Psi(K-c|d)$.
\end{proof}

\begin{lem}\label{1.6}
	If both $ed$ and $fc$ are not in the maximal matching, and $e$ and $f$ are both covered, then
	$$[\Psi_{ab}^{-\{ed,fc\}}(F|e,f)]_{[i]}=M[i,1]\cdot\Psi(K).$$
\end{lem}
\begin{proof}
As the proof of Lemma \ref{1.1}, we have $\Psi_{ab}^{-\{ed,fc\}}(F|e,f)=\Psi_{ab}(F-ed-fc|e,f)$, and each component of which can be computed as follows.

$[\Psi_{ab}(F-ed-fc|e,f)]_{[1]}=\Psi(F-ed-fc|e,f)=\Psi(G-d-c|e,f)\times\Psi(K)$.
Since $\Psi(G-d-c|e,f)=(1,0,0,0,0,0,0,0,0)\times\Psi_{ab}(G-d-c|e,f)=(1,0,0,0,0,0,0,0,0)\times f(M_1)\times f(M_2)\times\cdots\times f(M_{n-2})\times\Psi_{xy}(D|e,f)$, where $xy$ is the same as described in the proof of Lemma \ref{1.1}.
We note that $D$ is a subgraph of $G$, which is the union of $H_{n-1}$ and $H_{n}$ and delete vertices $c,d$.

For the three cases of the sharing edge between $H_{n-1}$ and $H_{n}$, we can check that
$\Psi_{xy}(D|e,f)=f(M_{n-1})\times f(M_{n})\times(1,0,0,0,0,0,0,0,0)^T$.
So $[\Psi_{ab}^{-e}(F|cf)]_{[i]}=M[i,1]\cdot\Psi(K)$.
\end{proof}

\begin{lem}\label{1.7}
	If both $ed$ and $fc$ are not in the maximal matching, and $e$ is in the matching while $f$ is not covered, then
	$$[\Psi_{ab}^{-\{ed,fc,f\}}(F|e)]_{[i]}=M[i,9]\cdot\Psi(K|c).$$
\end{lem}
\begin{proof}
As the proof of Lemma \ref{1.1}, we have $\Psi_{ab}^{-\{ed,fc,f\}}(F|e)=\Psi_{ab}^{-f}(F-ed|e)$, and each component of which can be computed as follows.

$[\Psi_{ab}^{-f}(F-ed|e)]_{[1]}=\Psi^{-f}(F-ed|e)=\Psi^{-f}(G-d-c|e)\times\Psi(K|c)$.
Since $\Psi^{-f}(G-d-c|e)=(1,0,0,0,0,0,0,0,0)\times\Psi_{ab}^{-f}(G-d-c|e)=(1,0,0,0,0,0,0,0,0)\times f(M_1)\times f(M_2)\times\cdots\times f(M_{n-2})\times\Psi_{xy}^{-f}(D|e)$, where $xy$ is the same as described in the proof of Lemma \ref{1.1}.
We note that $D$ is a subgraph of $G$, which is the union of $H_{n-1}$ and $H_{n}$ and delete vertices $c,d$.

For the three cases of the sharing edge between $H_{n-1}$ and $H_{n}$, we can check that
$\Psi_{xy}^{-f}(D|e)=f(M_{n-1})\times f(M_{n})\times(0,0,0,0,0,0,0,0,1)^T$.
So $[\Psi_{ab}^{-e}(F|cf)]_{[i]}=M[i,9]\cdot\Psi(K|c)$.
\end{proof}

\begin{lem}\label{1.8}
	If both $ed$ and $fc$ are not in the matching, and $f$ is in the matching while $e$ is not in the matching,
	$$[\Psi_{ab}^{-\{ed,fc,e\}}(F|f)]_{[i]}=M[i,8]\cdot\Psi(K|d).$$
\end{lem}
\begin{proof}
As the proof of Lemma \ref{1.1}, we have $\Psi_{ab}^{-\{ed,fc,e\}}(F|f)=\Psi_{ab}^{-e}(F-fc|f)$, and each component of which can be computed as follows.

$[\Psi_{ab}^{-e}(F-fc|f)]_{[1]}=\Psi^{-e}(F-fc|f)=\Psi^{-e}(G-d-c|f)\times\Psi(K|d)$.
Since $\Psi^{-e}(G-d-c|f)=(1,0,0,0,0,0,0,0,0)\times\Psi_{ab}^{-e}(G-d-c|f)=(1,0,0,0,0,0,0,0,0)\times f(M_1)\times f(M_2)\times\cdots\times f(M_{n-2})\times\Psi_{xy}^{-e}(D|f)$, where $xy$ is the same as described in the proof of Lemma \ref{1.1}.
We note that $D$ is a subgraph of $G$, which is the union of $H_{n-1}$ and $H_{n}$ and delete vertices $c,d$.

For the three cases of the sharing edge between $H_{n-1}$ and $H_{n}$, we can check that
$\Psi_{xy}^{-e}(D|f)=f(M_{n-1})\times f(M_{n})\times(0,0,0,0,0,0,0,1,0)^T$.
So $[\Psi_{ab}^{-\{ed,fc,e\}}(F|f)]_{[i]}=M[i,8]\cdot\Psi(K|d).$
\end{proof}

\begin{lem}\label{1.9}
	If both $ed$ and $fc$ are not in the matching, and neither $e$ nor $f$ are in the matching,
	$$[\Psi_{ab}^{-\{ed,fc,e,f\}}(F)]_{[i]}=M[i,5]\cdot\Psi(K|c,d).$$
\end{lem}
\begin{proof}
As the proof of Lemma \ref{1.1}, we have $\Psi_{ab}^{-\{ed,fc,e,f\}}(F)=\Psi_{ab}^{-\{e,f\}}(F)$, and each component of which can be computed as follows.

$[\Psi_{ab}^{-\{e,f\}}(F)]_{[1]}=\Psi^{-\{e,f\}}(F)=\Psi^{-\{e,f\}}(G-d-c)\times\Psi(K|c,d)$.
Since $\Psi^{-\{e,f\}}(G-d-c)=(1,0,0,0,0,0,0,0,0)\times\Psi_{ab}^{-\{e,f\}}(G-d-c)=(1,0,0,0,0,0,0,0,0)\times f(M_1)\times f(M_2)\times\cdots\times f(M_{n-2})\times\Psi_{xy}^{-\{e,f\}}(D)$, where $xy$ is the same as described in the proof of Lemma \ref{1.1}.
We note that $D$ is a subgraph of $G$, which is the union of $H_{n-1}$ and $H_{n}$ and delete vertices $c,d$.

For the three cases of the sharing edge between $H_{n-1}$ and $H_{n}$, we can check that
$\Psi_{xy}^{-\{e,f\}}(D)=f(M_{n-1})\times f(M_{n})\times(0,0,0,0,1,0,0,0,0)^T$.
So $[\Psi_{ab}^{-\{e,f\}}(F)]_{[i]}=M[i,5]\cdot\Psi(K|c,d)$.
\end{proof}


\subsection{\textbf{Proof of Theorem \ref{Hexagond}}}

\begin{proof}
We use the notations $H, G, F$ as depicted in Fig. \ref{Auxiliary graph}.
All the maximal matchings of $H$ can be classified into the following $4$ major cases and $9$ sub-cases based on whether edges $ed$ and $fc$ are contained.

$\textbf{Case 1:}$ Both $ed$ and $fc$ are in the maximal matching.

$\textbf{SubCase 1.1:}$ Specifically, we need to compute $\Psi(H|ed, fc)$

For the sake of convenience, here, we set $F^1:=F-a-b-e-d-f-c$, $G^1:=G-e-d-f-c-a-b$.
Clearly, $\Psi(H|ed,fc)=\Psi(H-e-d-f-c)=\Psi(G^1)$.
Since $F^1$ consists of two components $G^1$ and $K-d-c$,
$\Psi(F^1)=\Psi(G^1)\cdot\Psi(K-d-c)$.
On the other hand,
by Lemma \ref{1.1}, we have $\Psi(F^1)=\Psi(F-a-b|ed,fc)=[\Psi_{ab}(F|ed,fc)]_{[4]}=M[4,4]\cdot\Psi(K-d-c)$.
So $\Psi(H|ed,fc)=\Psi(G^1)=M[4,4]$.

$\textbf{Case 2:}$ $ed$ is in and $fc$ is not in the maximal matching.

$\textbf{SubCase 2.1:}$ If $f$ is in the maximal matching, then we compute $\Psi^{-fc}(H|ed,f)$.

For the sake of convenience, here, we set $F^2:=F-a-e-d-fc$, $G^2:=G-a-e-d-c$.
Clearly, $\Psi^{-fc}(H|ed,f)=\Psi(H-e-d-fc|f)=\Psi(G^2|f)$.
Since $F^2$ consists of two components $G^2$ and $K-d$,
$\Psi(F^2|f)=\Psi(G^2|f)\cdot\Psi(K-d)$.
On the other hand,
by Lemma \ref{1.2}, we have $\Psi(F^2|f)=\Psi^{-fc}(F-a|ed,f)=[\Psi_{ab}(F|ed,f)]_{[2]}=M[2,2]\cdot\Psi(K-d)$.
So $\Psi^{-fc}(H|ed,f)=\Psi(G^2|f)=M[2,2]$.

$\textbf{SubCase 2.2:}$ If $f$ is not in the maximal matching, then we compute $\Psi^{-\{fc,f\}}(H|ed).$

In this case, we set $F^3=F-a-e-d-f-b-w$, $G^3=G-a-e-d-f-b-w-c$.
Clearly, $\Psi^{-\{fc,f\}}(H|ed)=\Psi^{-f}(H-e-d)=\Psi(H-e-d-f|N_H(f))=\Psi(G^3|N_{G^3}(f))$.
Since $F^3$ consists of two connected components $G^3$ and $K-d$,
$\Psi(F^3|N_{F^3}(f))=\Psi(G^3|N_{G^3}(f))\cdot\Psi(K-d|c)$.
On the other hand,
by Lemma \ref{1.3}, we have $\Psi(F^3|N_{F^3}(f))=\Psi^{-f}(F-a|ed,b)=[\Psi_{ab}^{-f}(F|ed)]_{[6]}=M[6,6]\cdot\Psi(K-d|c)$.
So $\Psi^{-\{fc,f\}}(H|ed)=\Psi(G^3|N_{G^3}(f))=M[6,6]$.

$\textbf{Case 3:}$ $fc$ is in the maximal matching, $ed$ is not in the maximal matching.

$\textbf{SubCase 3.1:}$ If $e$ is in the maximal matching, we compute $\Psi^{-ed}(H|fc,e)$.

This is the symmetry case of Subcase 2.1,
similarly, we can obtaine that $\Psi^{-ed}(H|fc,e)=M[3,3]$.

$\textbf{SubCase 3.2:}$ If $e$ is not in the maximal matching, we compute $\Psi^{-\{ed,e\}}(H|fc)$.

This is the symmetry case of Subcase 2.2,
similarly, we can show that
$\Psi^{-\{ed,e\}}(H|fc)=M[7,7]$.

$\textbf{Case 4:}$ Neither $ed$ nor $fc$ is in the maximal matching.

$\textbf{SubCase 4.1:}$ If both $e$ and $f$ are in the maximal matching, then we compute $\Psi^{-\{fc,ed\}}(H|e,f)$.

For the sake of convenience, here, we set $F^4=F-ed-fc$, $G^4=G-d-c$.
Clearly, $\Psi^{-\{fc,ed\}}(H|e,f)=\Psi(H-ed-fc|e,f)=\Psi(G^4|e,f)$.
Since $F^4$ consists of two components $G^4$ and $K$,
$\Psi(F^4|e,f)=\Psi(G^4|e,f)\cdot\Psi(K)$.
On the other hand,
by Lemma \ref{1.6}, we have $\Psi(F^4|e,f)=\Psi^{-\{fc,ed\}}(F|e,f)=[\Psi_{ab}^{-\{fc,ed\}}(F|e,f)]_{[1]}=M[1,1]\cdot\Psi(K)$.
So $\Psi^{-\{fc,ed\}}(H|e,f)=\Psi(G^4|e,f)=M[1,1]$.

$\textbf{SubCase 4.2:}$ If $e$ is in and $f$ is not in the maximal matching, then we compute $\Psi^{-\{fc,ed,f\}}(H|e)$.

For the sake of convenience, we set $F^5=F-f-ed$, $G^5=G-f-c-d$.
Clearly, $\Psi^{-\{fc,ed,f\}}(H|e)=\Psi^{-f}(H-ed|e)=\Psi(H-ed-f|e, N_{H}(f))=\Psi(G^5|e,b,N_{G}(f)\setminus c)$.
Since $F^5$ consists of two components $G^5$ and $K$,
$\Psi(F^5|e,b,N_F(f))=\Psi(G^5|e,b,N_{G}(f)\setminus c)\cdot\Psi(K|c)$.
On the other hand,
by Lemma \ref{1.7}, we have $\Psi(F^5|e,b,N_F(f))=\Psi^{-\{ed,f\}}(F|e,b)=[\Psi_{ab}^{-\{ed,f\}}(F|e)]_{[9]}=M[9,9]\cdot\Psi(K|c)$.
So $\Psi^{-\{fc,ed,f\}}(H|e)=\Psi(G^5|e,b,N_{G}(f)\setminus c)=M[9,9]$.

$\textbf{SubCase 4.3:}$ If $f$ is in the maximal matching and $e$ is not in, then we compute  $\Psi^{-\{fc,ed,e\}}(H|f)$.

This is the symmetry case of Subcase 4.2,
similarly, we can show
$\Psi^{-\{fc,ed,d\}}(H|f)=M[8,8]$.

$\textbf{SubCase 4.4:}$ If neither $e$ nor $f$ is in the maximal matching, then we compute $\Psi^{-\{fc,ed,e,f\}}(H)$.

For the sake of convenience, here, we set $F^6=F-e-f$ and $G^6=G-e-f-d-c$.
Clearly, $\Psi^{-\{fc,ed,e,f\}}(H)=\Psi^{-\{e,f\}}(H)=\Psi(H-e-f|N_H(e),N_H(f))=\Psi(G^6|a,b,N_F(e)\setminus d, N_F(f)\setminus c)$.
It is easy to see that $F^6$ consists of two components $G^6$ and $K$.
So we have $\Psi(F^6|a,b,N_F(e),N_F(f))$\\
$=$$\Psi(G^6|a,b,N_F(e)\setminus d, N_F(f)\setminus c)\cdot\Psi(K|d,c)$.
On the other hand,
$\Psi(F^6|a,b,N_F(e),N_F(f))=\Psi^{-\{e,f\}}(F|a,b)=[\Psi_{ab}^{-\{e,f\}}(F)]_{[5]}=M[5,5]\cdot\Psi(K|c,d)$ by Lemma \ref{1.9}.
So $\Psi^{-\{fc,ed,f,e\}}(H)=\Psi(G^6|a,b,N_F(e)\setminus d, N_F(f)\setminus c)=M[5,5]$.

In summary, we have
\begin{equation}\label{4}
	\Psi(H)=\sum_{i=1}^{9}M[i,i]=tr(M)
\end{equation}
\end{proof}

It is clearly that for the hexagonal ring $H$, selecting any hexagon as the first face $H_1$ leads to the same conclusion.
The table \ref{table-hexagon-ring} lists the results of several hexagonal rings.

\begin{table}
\centering
\includegraphics[width=0.90\linewidth]{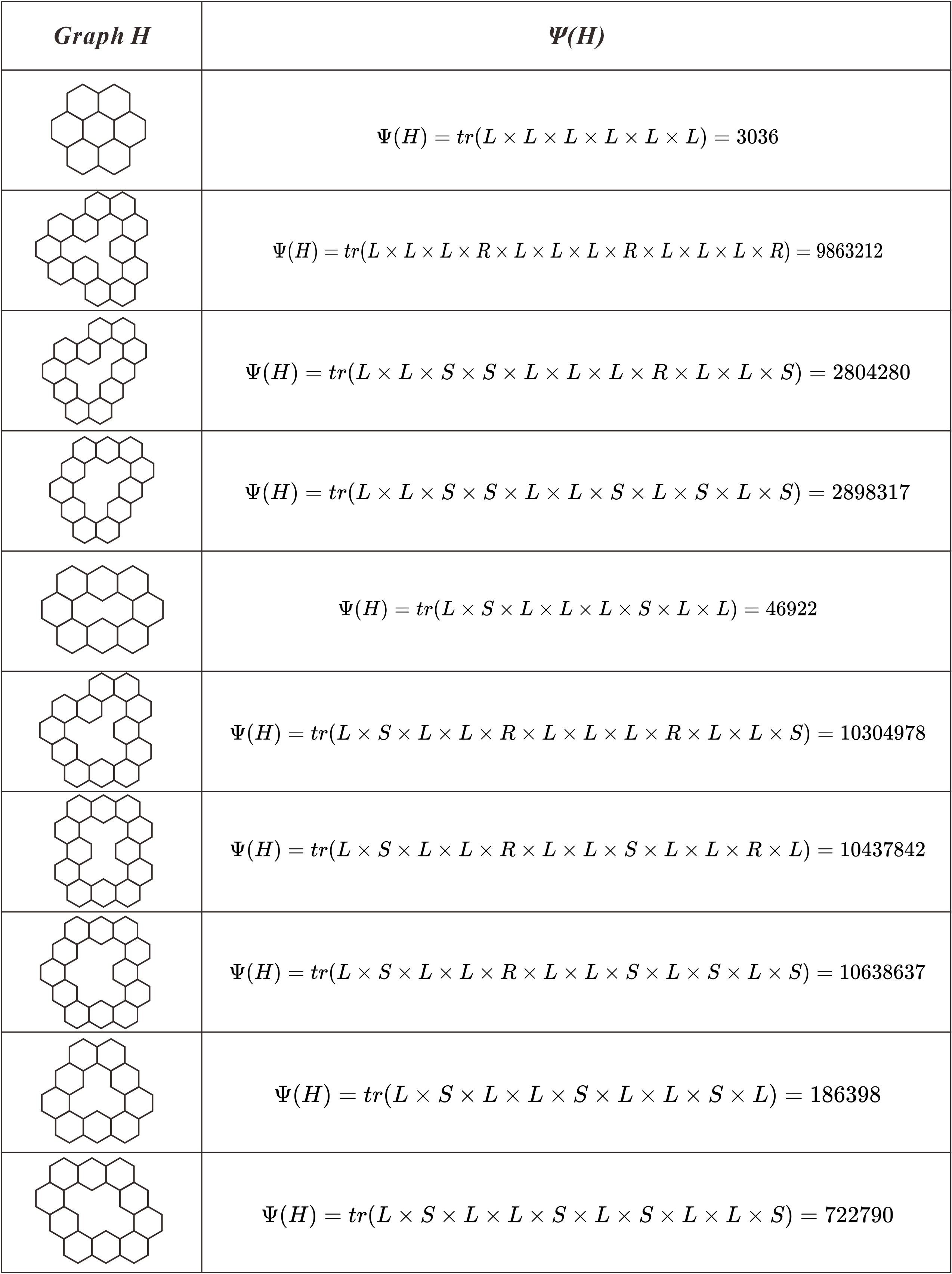}
\caption{The number of maximal matchings for several hexagonal rings.}\label{table-hexagon-ring}
\end{table}

\section{Polygon Rings}
\begin{lem}\label{process-polygon-chain}
Let $Z$ be a polygon chain as depicted in Fig. \ref{fig:polygon} and $J$ is an $m$-length face of $Z$.
If there are $i$ edges between $ab$ and $dc$ on the face $J$ along the clockwise direction,
then
$$\Psi_{ab}(Z)=T_{m,i}\times \Psi_{dc}(K),$$
where $T_{m,i}$ is a $9\times9$ matrix which can be obtained by the Algorithm \ref{algorithm}.
\end{lem}
\begin{proof}
By the process of the Algorithm \ref{algorithm}, this conclusion clearly holds.
\end{proof}
\begin{figure}[h]
	\centering
	\includegraphics[width=0.35\linewidth]{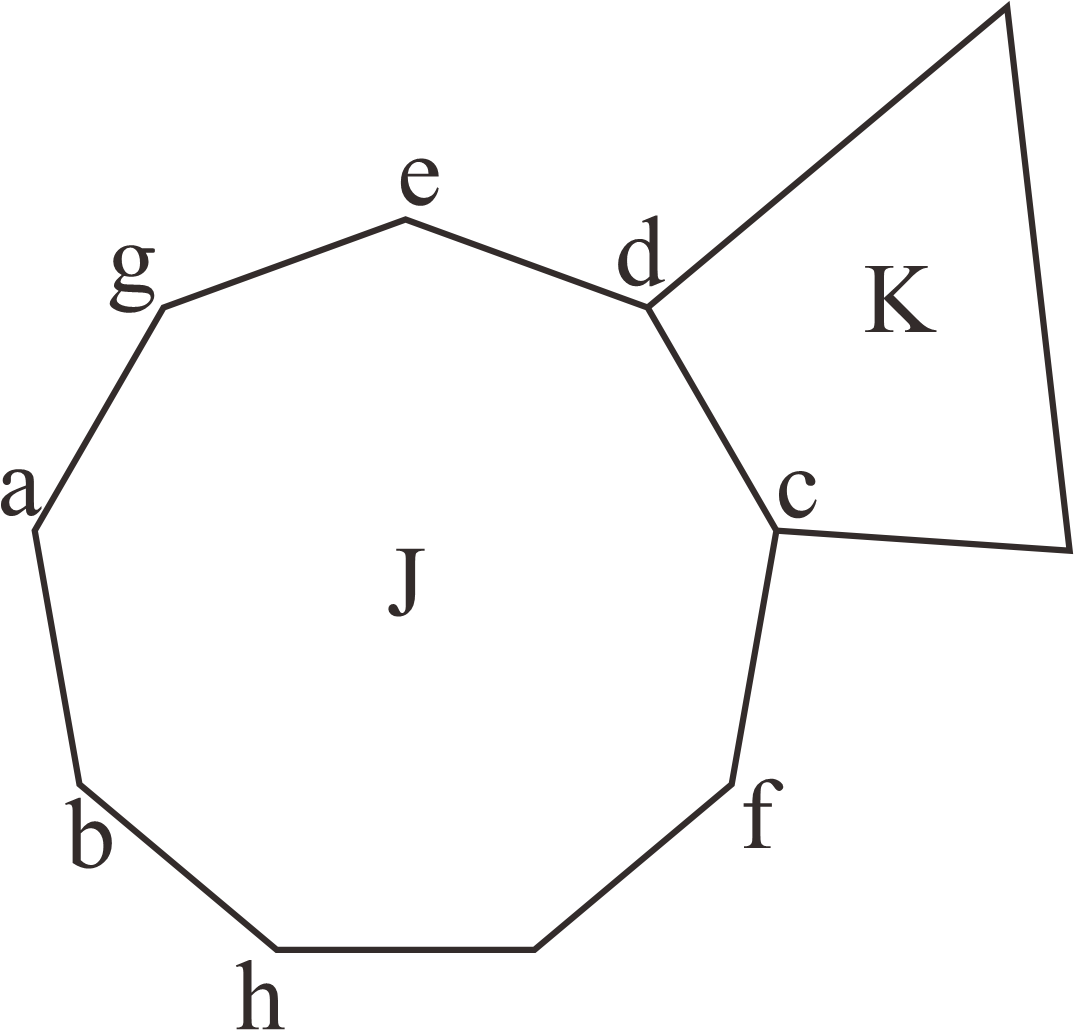}
	\caption{Graph $Z$}
	\label{fig:polygon}
\end{figure}
\begin{alg}\label{algorithm}
\textbf{(The algorithm to determine $T_{m,i}$)}

1. Input $m,i$

2. Create an $m$-sided polygon $J$ based on the value of $m$, arbitrarily select one of the sides as $ab$, and determine the side $dc$ according to the connection type $t(m,i)$.

3. Subtract $dc$ from the polygon $J$, add a new neighbor $k_1$ for $d$, and add a new neighbor $k_2$ for $c$ to obtain the new graph $J'$.

4. Create a $9\times 9$ all-zero matrix $T$ to store the results, initialize variable $x=0$.

5. For variable $x$ ranging from 1 to 9, create empty sets $W, R = \{\}, C = \{\}$, and perform the following operations:

(5.1) When $x=1$, set $W$ to all maximal matchings of $J'$.

(5.2) When $x=2$, set $W$ to all maximal matchings of $J'-a$, and set $R$ to $\{a\}$.

(5.3) When $x=3$, set $W$ to all maximal matchings of $J'-b$, and set $R$ to $\{b\}$.

(5.4) When $x=4$, set $W$ to all maximal matchings of $J'-a-b$, and set $R$ to $\{a,b\}$.

(5.5) When $x=5$, set $W$ to all maximal matchings of $J'$ under the condition that points $a$ and $b$ are covered by the matching, and set $C$ to $\{a,b\}$.

(5.6) When $x=6$, set $W$ to all maximal matchings of $J'-a$ under the condition that point $b$ is covered by the matching, and set $R$ to $\{a\}$, set $C$ to $\{b\}$.

(5.7) When $x=7$, set $W$ to all maximal matchings of $J'-b$ under the condition that point $a$ is covered by the matching, and set $R$ to $\{b\}$, set $C$ to $\{a\}$.

(5.8) When $x=8$, set $W$ to all maximal matchings of $J'$ under the condition that point $a$ is covered by the matching, and set $C$ to $\{a\}$.

(5.9) When $x=9$, set $W$ to all maximal matchings of $J'$ under the condition that point $b$ is covered by the matching, and set $C$ to $\{b\}$.

(5.10) Traverse each matching $W_i(V,E)$ in $W$, create variable $y$, and create set $V'=V\cup R$, $W_i'=W_i-k_1-k_2$. Perform the following operations:
The set $R$ stores the subtracted points, that is, the points where $J$ has been matched with adjacent graphs other than $K$.
The set $V$ stores the points covered in the matching of $J'-R$.
Therefore, $V'$ stores the points covered in the matching of $J'$.

(5.10.1) If points $k_1, e, k_2, f \in V'$, set $y$ to 1 and jump to (5.10.10).\par
$k_1$ is in the matching if and only if $k_1d$ is in the matching, indicating that $de$ is not in the matching; Meanwhile, $e$ is in the matching;
$k_2$ is in the matching if and only if $k_2c$ is in the matching, indicating that $fc$ is not in the matching; Meanwhile, $f$ is in the matching;
If and only if the matching $W_i$ meets the above requirements, $\Psi(F-R|C,E(W_i'))=\Psi(K)$.

(5.10.2) If points $k_2, f \in V'$ and point $k_1 \notin V'$, set $y$ to 2 and jump to (5.10.10).\par
$k_1$ is not in the matching if and only if $de$ is in the matching;
$k_2$ is in the matching if and only if $k_2c$ is in the matching, indicating that $fc$ is not in the matching; Meanwhile, $f$ is in the matching;
If and only if the matching $W_i$ meets the above requirements, $\Psi(F-R|C,E(W_i'))=\Psi(K-d)$.

(5.10.3) If points $k_1, e \in V'$ and point $k_2 \notin V'$, set $y$ to 3 and jump to (5.10.10).\par
$k_1$ is in the matching if and only if $k_1d$ is in the matching, indicating that $de$ is not in the matching; Meanwhile, $e$ is in the matching;
$k_2$ is not in the matching if and only if $fc$ is in the matching;
If and only if the matching $W_i$ meets the above requirements, $\Psi(F-R|C,E(W_i'))=\Psi(K-c)$.

(5.10.4) If points $k_1, k_2 \notin V'$, set $y$ to 4 and jump to (5.10.10).\par
$k_1$ is not in the matching if and only if $de$ is in the matching;
$k_2$ is not in the matching if and only if $fc$ is in the matching;
If and only if the matching $W_i$ meets the above requirements, $\Psi(F-R|C,E(W_i'))=\Psi(K-c-d)$.

(5.10.5) If points $e, f \notin V'$, set $y$ to 5 and jump to (5.10.10).\par
$e$ is not in the matching if and only if $k_1d$ is in the matching, indicating that $de$ is not in the matching;
$f$ is not in the matching if and only if $k_2c$ is in the matching, indicating that $fc$ is not in the matching;
If and only if the matching $W_i$ meets the above requirements, $\Psi^{-\{de,fc\}}(F-R|C,E(W_i'))=\Psi(K|d,c)$.

(5.10.6) If point $k_1, f \notin V'$, set $y$ to 6 and jump to (5.10.10).\par
$k_1$ is not in the matching if and only if $de$ is in the matching;
$f$ is not in the matching if and only if $k_2c$ is in the matching, indicating that $fc$ is not in the matching;
If and only if the matching $W_i$ meets the above requirements, $\Psi^{-fc}(F-R|C,E(W_i'))=\Psi(K-d|c)$.

(5.10.7) If point $k_2, e \notin V'$, set $y$ to 7 and jump to (5.10.10).\par
$k_2$ is not in the matching if and only if $fc$ is in the matching;
$e$ is not in the matching if and only if $k_1d$ is in the matching, indicating that $de$ is not in the matching;
If and only if the matching $W_i$ meets the above requirements, $\Psi^{-de}(F-R|C,E(W_i'))=\Psi(K-c|d)$.

(5.10.8) If points $k_2, f \in V'$ and point $e \notin V'$, set $y$ to 8 and jump to (5.10.10).\par
$k_2$ is in the matching if and only if $k_2c$ is in the matching, indicating that $fc$ is not in the matching; Meanwhile, $f$ is in the matching;
$e$ is not in the matching if and only if $k_1d$ is in the matching, indicating that $de$ is not in the matching;
If and only if the matching $W_i$ meets the above requirements, $\Psi^{-de}(F-R|C,E(W_i'))=\Psi(K|d)$.

(5.10.9) If points $k_1, e \in V'$ and point $f \notin V'$, set $y$ to 9 and jump to (5.10.10).\par
$k_1$ is in the matching if and only if $k_1d$ is in the matching, indicating that $de$ is not in the matching; Meanwhile, $e$ is in the matching;
$f$ is not in the matching if and only if $k_2c$ is in the matching, indicating that $fc$ is not in the matching;
If and only if the matching $W_i$ meets the above requirements, $\Psi^{-fc}(F-R|C,E(W_i'))=\Psi(K|c)$.

(5.10.10) Increment the element at the $x$-th row and $y$-th column of $T$, i.e., set $T[x, y]:= T[x, y] + 1$.

6. Assign the value of $T$ to $T_{m,i}$, and output the result $T_{m,i}$.
\end{alg}

For polygons $J$ with different numbers of edges $m$, corresponding to different connection modes $t(m,i)$, different matrices $T_{m,i}$ are generated by Algprithm \ref{algorithm}. Table \ref{transition matrixes} lists all the cases for $m=5, 6, 7, 8$.

\begin{table}
	\includegraphics[width=1\linewidth]{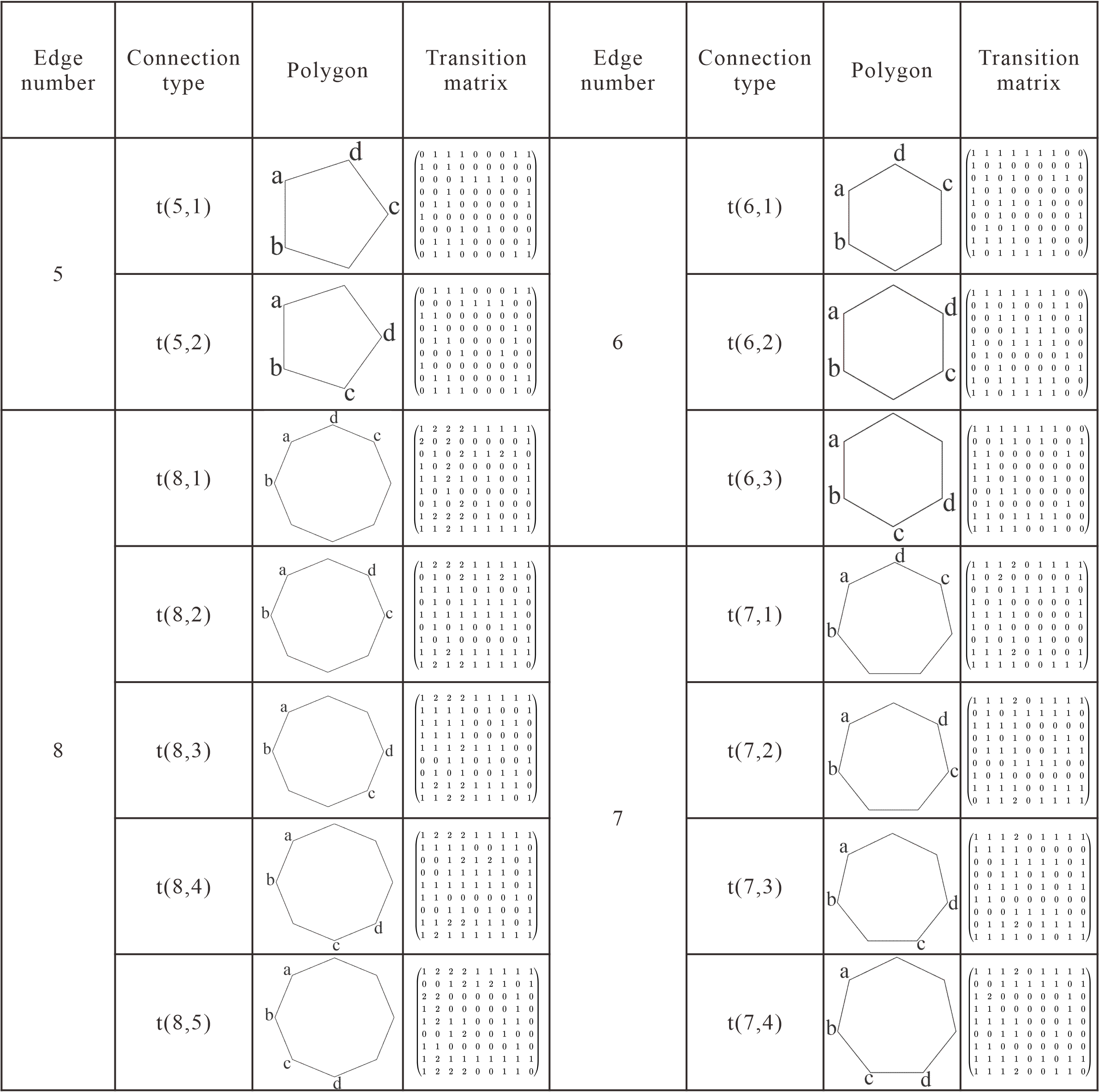}
    \caption {The transition matrixes.}\label{transition matrixes}
\end{table}
\begin{thm}\label{polygon-chain}
Let $Z$ be a polygon chain with $n$ faces and having type $t(s_1,*)M_2\cdots M_{n-1}t(s(n),*)$ where for $i\in\{2,3,\ldots,n-1\}$, $M_i=t(s_i,1), t(s_i,2), \ldots t(s_i,s_i-4)$ or  $t(s_i, s_i-3)$. Then the number of maximal matchings of $Z$ is
 $$\Psi(Z)=X\times T_{s_1,1}\times\prod\limits_{i=2}^{n-1}(f(M_i))\times T_{s_n,1}\times Y,$$
	where $f(M_{i})=T_{s_i,1}$ if $M_i=t(s_i,1)$, $f(M_{i})=T_{s_i,2}$ if $M_i=t(s_i,2)$, $\ldots$, $f(M_{i})=T_{s_i,s_i-3}$ if $M_i=t(s_i,s_i-3)$, and the matrix $T_{s_i,j}$ is depicted in Lemma \ref{process-polygon-chain}, $j\in\{1, 2\cdots, s_i-3\}$, and $X=(1,0,0,0,0,0,0,0,0)$, $Y=(1,1,1,1,1,0,0,1,1)^T$.
\end{thm}
\begin{proof}
By using Lemma \ref{process-polygon-chain} repeatedly, and note that the maximal matching vector of a complete graph with $2$ vertices is $(1,1,1,1,1,0,0,1,1)^T$, the conclusion can be obtained as the proof of Theorem \ref{shi-deng-process}.
\end{proof}
Now, we can show Theorem \ref{Ploygon} which is similar to the proof of Theorem \ref{Hexagond} but more tedious. So we omit its proof here.

Using the matrices above, for the polygon ring $H$ in the Fig.\ref{fig:polygon-ring}, its maximal matching number can be calculated based on the connection structures of its constituent polygons, as follows:
$$
\Psi (H) = \text{tr}(T_{(7,3)} \cdot T_{(5,2)} \cdot T_{(8,4)} \cdot T_{(6,3)} \cdot T_{(5,2)} \cdot T_{(6,2)} \cdot T_{(8,3)} \cdot T_{(6,3)} \cdot T_{(6,3)}) = 481614.
$$

\section{Acknowledgements}
This research was supported by NSFC (grant no. 11901458) and the Natural Science Foundation of Shaanxi Province (grant no. 2024JC-YBQN-0053).
\bibliography{papers}

\end{document}